\documentclass[a4paper]{article}
\usepackage[left=4cm,right=2.5cm,
    top=4cm,bottom=4cm,bindingoffset=0cm]{geometry}

\usepackage[utf8x]{inputenc}
\usepackage[T2A]{fontenc}

\usepackage{pgfplots}
\usepackage{mathrsfs}
\usetikzlibrary{arrows.meta}
\usetikzlibrary{arrows}
\usepackage{bbding}

\usepackage{caption}
\usepackage{subcaption}

\usepackage[main=english,russian]{babel}

\usepackage{bm}

\usepackage{verbatim}
\usepackage{comment}

\usepackage{amsmath}
\usepackage{amssymb}
\usepackage{amsfonts}
\usepackage{amsthm}

\usepackage{mathtools}

\usepackage{tikz-cd}

\usepackage{hyperref}
\usepackage{graphicx}
\hypersetup{
    colorlinks=true,
    linkcolor=blue,
    filecolor=magenta,
    urlcolor=cyan,
}

\usepackage{authblk}

\theoremstyle{definition}

\theoremstyle{remark}

\theoremstyle{definition}

\theoremstyle{plain}

\newtheorem{proposition}{Proposition}
\newtheorem{theorem}{Theorem}

\theoremstyle{plain}
\newtheorem{lemma}{Lemma}
\theoremstyle{remark}
\newtheorem{rem}{Remark}

\renewcommand{\Re}{\operatorname{Re}}
\renewcommand{\Im}{\operatorname{Im}}

\newcommand{\RR}{\mathbb{R}}
\newcommand{\CC}{\mathbb{C}}

\newcommand{\PC}{\mathcal{P}}

\newcommand{\dd}{\mathrm{d}}

\title{How close are cone singularities on a random flat surface?}
\author[1,2]{A. Rukhovich}
\date{}

\begin{document}


    \maketitle

    \begin{abstract}
        We study the shortest geodesics on flat cone spheres, i.e.\ flat metrics on the sphere
        with conical singularities.
        The length of the shortest geodesic between two singular points can be treated as a function
        on the moduli space of flat cone spheres with prescribed angle defects.
        We prove a recurrent relation on the distribution of this function with respect
        to Thurston's volume form on the moduli space.
    \end{abstract}

    \section{Introduction}\label{sec:intro}

    A \textit{flat cone sphere} is a metric on the 2-sphere which is locally isometric to the Euclidean plane
    except at finitely many points, where it has conical singularities.
    This means that a neighborhood of the point is isometric to a Euclidean cone with total angle $2\pi-\alpha$
    for some $\alpha<2\pi$.
    The number $\alpha$ is called \textit{angle defect}.

    Using Gauss-Bonnet theorem it is easy to see that the sum of angle defects always equals $4\pi$.
    Let $\alpha = (\alpha_1,\dots,\alpha_n)$ be a list of nonzero real numbers less than $2\pi$ such that
    $\alpha_1+\dots+\alpha_n = 4\pi$.
    Denote by $\mathcal{M}_\alpha$ the
    space of isometry classes of flat cone spheres with angle defects
    $\alpha_1,\dots,\alpha_n$ and unit area.

    The theory of such spaces was developed for the case where all $\alpha_i$ are positive by Thurston
    in his work~\cite{thurston1998shapes} extending the study of
    Deligne and Mostow~\cite{deligne1986monodromy}.
    Thurston proved that $\mathcal{M}_\alpha$ carries a natural complex hyperbolic metric and has a finite volume
    with respect to this metric (see also~\cite{schwartz2015notes}, ~\cite{nguyen2010triangulations}).
    A more general situation allowing negative angle defects was conducted by Veech~\cite{veech1993flat}.
    Moreover, this work deals with moduli spaces of flat metrics with cone singularities
    on orientable surfaces of any genus.
    In this case a similar construction instead of giving a complex hyperbolic metric gives
    the structure of a space locally modelled by $U(p, q) / U(0, 1)\times U(p, q-1)$ for some positive integers $p, q$.
    Veech proved that the moduli spaces of flat metrics with (not necessarily positive angle defect)
    cone singularities is of finite volume for almost all lists of angle defects~\cite[Theorem 0.25]{veech1993flat}.




    Denote by $\mathrm{Vol}$ Thurston's hyperbolic volume form on $\mathcal{M}_\alpha$.
    We will also use $\mathrm{Vol}$ for the volume function on (the sigma-algebra of Borel) subsets of $\mathcal{M}_\alpha$.
    Thurston proved that the total volume $\mathrm{Vol}(\mathcal{M}_\alpha)$ is finite for each $\alpha$.
    The exact value of this volume was calculated by McMullen~\cite{mcmullen2017gauss} .

    Despite the known volume of the moduli space, not much is known about the behaviour of a random flat cone sphere.
    In the study of moduli space of hyperbolic metrics on surfaces of positive genus the common characteristic of
    a metric is the length of a systole -- a closed geodesic in given free homotopy class.
    This was used by Mirzakhani to obtain a recurrent formula for Weil-Petersson volume of moduli space of
    hyperbolic Riemann surfaces~\cite{mirzakhani2007simple}.
    In the study of translation surfaces there is an approach due to Masur et al~\cite{masur2022expected} 
    investigating the covering radius, that is the largest radius of an immersed disk.
    They prove some bounds on the asymptotics of expected value of this radius when the genus goes to infinity.
    We deal with another natural characteristic of a flat metric: the \textit{length function}, i.e.~the
    length of the shortest geodesic between two fixed singular points or simply the distance between these points.
    This is also applicable for studying translation surfaces, in this direction there is a well-known result than 
    $V(\epsilon)\simeq c\epsilon^2 \text{ as } \epsilon\to 0$ where $V(\epsilon)$ is the volume of set of 
    surfaces where the length function is less than $\epsilon$.
    There are also several studies proving $V_2(\epsilon)\lesssim \epsilon^2$ for volume of the set of surfaces having two
    non-parallel geodesics of length less than $\epsilon$ in the strata of moduli spaces of translation
    surfaces and affine submanifolds of these strata~\cite{masur1991hausdorff, avila2013sl, nguyen2012volumes}.

    The aim of our study is finding the distribution of the length function $l$ on $\mathcal{M}_\alpha$ where $l(\rho)$ is the distance between
    two fixed singularities with angle defects $\alpha_1, \alpha_2$ in the flat cone sphere $\rho$.


    To state the result let us describe the setup more precisely.
    We distinguish two of the cone singularities on the flat cone sphere and
    for $\varphi = (\varphi_1, \varphi_2),\ \alpha = (\alpha_1, \dots, \alpha_n)$ denote
    $\mathcal{M}_{\varphi, \alpha} \coloneqq \mathcal{M}_{(2\pi-\varphi_1,2\pi-\varphi_2, \alpha_1,\dots,\alpha_n)}$.
    Therefore $\varphi_1, \varphi_2$ are the total angles of the distinguished cone singularities.
    The conditions on $\varphi, \alpha$ will be
    $\varphi_1+\varphi_2=\alpha_1+\dots+\alpha_n;\ 0<\alpha_i<2\pi$.
    For flat cone spheres in $\mathcal{M}_{\varphi, \alpha}$ we call
    the first two singular points \textit{distinguished}.

    The length of the shortest geodesic between distinguished points determines a function on $\mathcal{M}_{\varphi, \alpha}$.
    More precisely, to calculate the \textit{length function $\bm{l}$} of $m\in \mathcal{M}_{\varphi, \alpha}$
    one has to take a flat cone sphere representing $m$ such that its total area is unit and
    denote $\bm{l}(m)$ to be the distance between the distinguished singularities in this flat cone sphere.
    The length function $\bm{l}:\mathcal{M}_{\varphi, \alpha}\to\mathbb{R}_+$ is piecewise smooth, and we can consider
    the integration of the volume form along fibers of $\bm{l}$.
    This yields a $1$-form $\bm{l}_* \mathrm{Vol}_{\varphi,\alpha}$ on $\mathbb{R}_+$
    called \textit{length distribution for $\mathcal{M}_{\varphi, \alpha}$}.
    The formula for its density function $\rho_{\varphi, \alpha}(l)$ is our main result.

    To simplify the resulting formula let us introduce the
    \textit{area function $\mathbf{a}:\mathcal{M}_{\varphi, \alpha}\to\mathbb{R}_+$} to be $\mathbf{a}(m) = 1/\bm{l}^2(m)$.
    It can be equivalently defined as area of the flat cone sphere representing $m$ blown with homothety
    such that the distance between the distinguished points is unit.
    Finding the density function $\rho_{\varphi, \alpha}(l)$ of the distribution of $\bm{l}$ is equivalent to
    finding the analogous density function $f_{\varphi, \alpha}(a)$ for the distribution of $\mathbf{a}$.
    Precisely,
    \[l^{-3}\rho_{\varphi, \alpha}(l) = 2f_{\varphi, \alpha}(1/l^2).\]

    For list $\alpha = (\alpha_1,\dots, \alpha_n)$ denote by $\#\alpha, \Sigma\alpha$ respectively the number and the sum of elements,
    i.e. $\#\alpha=n,\ \Sigma\alpha=\alpha_1+\dots+\alpha_n$.

    \begin{theorem}
        \label{thm1}
        Let $f(\varphi, \alpha, a)$ be the density function of $\mathbf{a}_*\mathrm{Vol}_{\varphi,\alpha}$ where  $\mathbf{a}:\mathcal{M}_{\varphi,\alpha}\to\RR_+$ is the area function.
        Then the upper bound of the support of $f(a) = f(\varphi, \alpha, a)$ is infinite if $\varphi_1 + \varphi_2 \ge 2\pi$ and $\frac{1}{\cot\frac{\varphi_1}{2} + \cot\frac{\varphi_2}{2}}$ when $\varphi_1 + \varphi_2 < 2\pi$.

        On the interval from zero to this upper bound the function $f(a)$ satisfies the recurrent relation
        \[\begin{multlined}
              \left(2q(\varphi)-\frac{n}{a}\right)\,f(\varphi, \alpha, a)'_a + \left(1-q(\varphi)a\right)\,f(\varphi, \alpha, a) =\\=
              \sum_{(\widehat{\alpha}, \widetilde{\alpha})} \int_0^a dx \int_{I(\varphi, \alpha_1, \alpha_2)}d\beta\
              \frac{x^{\#\widehat{\alpha}}(a-x)^{\#\widetilde{\alpha}}}{a^{\#\alpha}}
              r(\beta, \Sigma\widehat{\alpha}-\beta, \varphi_1-\beta, \Sigma\widetilde{\alpha}-\varphi_1+\beta)
              \cdot \\ \cdot
              f\big((\beta, \Sigma\widehat{\alpha}-\beta), \widehat{\alpha}, x\big)\
              f\big((\varphi_1-\beta, \Sigma\widetilde{\alpha}-\varphi_1+\beta), \widetilde{\alpha}, a-x\big)
        \end{multlined}\]
        where the sum is over shuffles of the list $\alpha$ into two lists
        $\widehat{\alpha} = (\alpha_{i_1},\dots,\alpha_{i_{n_1}}),
        \widetilde{\alpha} = (\alpha_{j_1},\dots,\alpha_{j_{n_2}})$ and the integration is over
        the interval $I(\varphi, \alpha_1, \alpha_2)$ from
        $\min(0, \varphi_1-\Sigma\widetilde{\alpha})$ to $\max(\Sigma\widehat{\alpha}, \varphi_1)$.
        (Note that the interval $I(\varphi, \alpha_1, \alpha_2)$ may be empty.)
        Here
        \begin{gather*}
            q(\varphi) = \cot\frac{\varphi_1}{2}+\cot\frac{\varphi_2}{2},\\
            r(\gamma_1, \gamma_2, \gamma_3, \gamma_4) =
            \cot\frac{\gamma_1}{2}+\cot\frac{\gamma_2}{2}+
            \cot\frac{\gamma_3}{2}+\cot\frac{\gamma_4}{2}.
        \end{gather*}
        The starting values for this recurrence are delta-functions for $1$-element $\alpha = (\alpha_1)$,
        namely,
        \[f(\varphi, (\alpha_1), a) = \delta\left(a-\frac{1}{\cot\frac{\varphi_1}{2} +
        \cot\frac{\varphi_2}{2}}\right).\]
    \end{theorem}





    \bigskip
    \textbf{Acknowledgements.} I am grateful to prof.~Alexander Gaifullin for his constant attention to the work and
    valuable conversations.
    Also, I would like to thank Vladislav Cherepanov, Nikita Klemyatin and Denis Gorodkov for discussions and help with
    the text preparation.

    \section{Polygonal coordinates}\label{sec:polygonal_coordinates}
    In this section we review the construction of coordinates on the labeled moduli space $\mathcal{M}_\alpha$ of
    flat cone spheres with fixed angle defects $\alpha=(\alpha_1,\dots,\alpha_n)$ up to homothety.
    This construction can be found in~\cite{thurston1998shapes, schwartz2015notes}.

    Let $X = (X_1,\dots,X_n)$ be $n$ points on the sphere $S^2$ and let us assign angle defects $\alpha_i$
    to points $X_i$ respectively.
    The points of $\mathcal{M}_\alpha$ will therefore correspond to homothety classes of flat metrics on $S^2$
    having singularity of defect $\alpha_i$ at each point $X_i$ modulo orientation-preserving
    homeomorphisms of $S^2$ leaving $X$ pointwise constant.

    Let us take the universal cover $p:U\to S^2\setminus x$ and its metric completion $\widehat{U}$.
    Note that the covering map $p$ can be continued to a map $\widehat{U}\to S^2$ so that the
    singular points of $\widehat{U}$ project to points in $X$.
    The metric of a flat cone sphere can be lifted from $S^2$ to $U$ making
    a $\pi_1(S^2\setminus X)$-invariant flat metric on $U$.

    Since $U$ is homeomorphic to a disk, flat metrics can be obtained as pullbacks of the
    Euclidean metric on the plane $\mathbb{C}=\mathbb{R}^2$.
    A map from $U$ to $\mathbb{C}$ will be called \textit{developing map} if the induced
    metric on $U$ coincides with the metric lifted to $U$ from some flat metric with cone singularities on $S^2$.

    To be more precise, denote by \textit{developing map} corresponding to the list $\alpha$
    a map $g:U\to \mathbb{C}$ which is orientation-preserving and locally isometric with respect to
    metric on $U$ pulled from some flat cone sphere with angle defects $\alpha$.
    Equivalently, these developing maps can be characterized by the following properties:
    \begin{itemize}
        \item $g$ can be continued to $\widehat{U}$,
        \item $g$ is orientation-preserving (with respect to fixed orientation on $U$ and standard orientation on the plane),
        \item the metric $g^*(|dz|^2)$ is $\pi_1(S^2\setminus X)$-equivariant, where $z$ is standard complex coordinate in the plane,
        \item for a loop $\gamma_k$ in $S^2\setminus X$ winding around exactly one point $X_k$ in positive direction
            the corresponding element in $\pi_1(S^2\setminus X)$ transforms the form $g^*(dz)$ to $e^{i\alpha_k}g^*(dz)$.
    \end{itemize}

    Then any flat cone sphere gives a developing map, though not unique.
    The flat cone sphere corresponding to a developing map $g$ is defined by the values of $g$ on the singular points
    of $\widehat{U}$ (cf.\ \cite{thurston1998shapes}).
    Denote these singular points as $\widehat{x}$, then the list $(g(X))_{X\in\widehat{x}}$ is the data to
    construct a flat cone sphere.
    The space of such lists corresponding to developing maps taken up to plane translation is called the
    Teichmüller space $\mathcal{T}_\alpha$ (Thurston called it the space of cocycles).
    It is proven (\cite{thurston1998shapes,nguyen2010triangulations}) that one can choose $n-2$ independent
    coordinates of the form $g(P_1)-g(P_2)$ for $P_1, P_2\in \widehat{x}$ and all other values of this form
    would be certain linear combinations of these coordinates.
    This turns the space $\mathcal{T}_\alpha$ into a complex affine manifold (by this we mean the existence of an atlas
    with affine transition maps).

    Now we want to pass from the coordinates on Teichmüller space $\mathcal{T}_\alpha$ to coordinates on the
    moduli space $\mathcal{M}_\alpha$.
    The latter is locally modelled as $\mathcal{T}_\alpha / \mathbb{C}^\times$, so the ratios between the
    independent coordinates on $\mathcal{T}_\alpha$ will be locally well-defined functions on $\mathcal{M}_\alpha$.
    There will be $n-3$ independent ratios, which provide a coordinate system on $\mathcal{M}_\alpha$.
    Note that despite the fact that the constructed coordinates on Teichmüller space are global coordinates,
    the ones on the moduli space are only local coordinates.
    The Hermitian form on the Teichmüller space defines a homogeneous Hermitian form on the moduli space.
    In case of all positive angle defects the former Hermitian form is of signature $(1, n-3)$ and the
    moduli space is locally modelled by the complex hyperbolic space.
    When some of the cone singularities have negative angle defects, the signature may be
    different~\cite{sauglam2022signature} but the construction still defines a volume form on $\mathcal{M}_\alpha$.

    Any developing map $t\in \mathcal{T}_\alpha$ defines a unique flat cone sphere.
    Let $\mathbf{A}(t)$ be the area of this sphere.
    The function $\mathbf{A}(t)$ can be written in coordinates $z = (z_1, \dots, z_{n-2})$ on $\mathcal{T}_\alpha$
    as $\overline{z} H z^T$.
    Then on $\mathcal{T}_\alpha$ we have a $\mathbb{C}^\times$-invariant form
    \[\frac{1}{2}\partial\overline{\partial}\log\mathbf{A}.\]
    In the homogeneous coordinates $y = (y_1, \dots, y_{n-3}) = \left(\frac{z_1}{z_{n-2}}, \dots, \frac{z_{n-3}}{z_{n-2}}\right)$
    we can write the corresponding form on $T_\alpha / \mathbb{C}^\times$ as
    \[
        \frac{(d\overline{y}, 0)H(dy, 0)^T (\overline{y}, 1)H(y, 1)^T -
            (d\overline{y}, 0)H(y, 1)^T(\overline{y}, 1)H(dy, 0)^T}{((\overline{y}, 1)H(y, 1)^T)^2}
    \]
    which defines a Hermitian form on $\mathcal{M}_\alpha$.
    The corresponding volume form is
    \[\frac{\det H}{(2i)^{n-3}((\overline{y}, 1)H(y, 1)^T)^{n-2}}dy_1\wedge
    d\overline{y}_1\wedge\dots\wedge dy_{n-3}\wedge d\overline{y}_{n-3}.\]

    We now describe an explicit choice of coordinates $z_1, \dots, z_{n-2}$ on $\mathcal{T}_\alpha$.
    Take the vertex $X_n\in x$ and consider some disjoint non-selfintersecting paths in $S^2$
    from $X_n$ to all the other vertices.
    The complement in $S^2$ of these paths is a disk $D$ and the complement of preimages
    of all paths in $U$ is a disjoint union of disks projecting to $D$.
    Take one of these disks in $U$, see Fig.~\ref{fig:pol}.
    Its boundary has $2(n-1)$ points of $\widehat{x}$, one point from preimage
    of each $X_i$ for $i = 1,\dots,n-1$ and $n-1$ points from preimage of $X_n$ in alternating order.
    Denote the latter points $P_1, \dots, P_{n-1}$.

    \begin{rem}\label{rem:noconfuse}
        The Figure~\ref{fig:pol} might seem quite misleading by its simplicity.
        Firstly, it corresponds the situation when the chosen paths from $X_n$ to other vertices
        are straight line segments.
        It is not generally true that one can choose such a collection of disjoint straight
        line segments \footnote{Though it is true for the Thurston's case when all angle defects $\alpha_i$ are positive.}.
        Moreover, we will need some special cyclic order of these paths at $X_n$, which makes the
        existence of the desired straight line segments more challenging.
        Happily, we do not require our paths to be straight, so the general picture may look more complicated.
        Secondly, it corresponds to the situation where the disk is isometric to some disk in the plane.
        In general, a disk with flat metric is not necessary isometric to a subset of a plane.
        For example, this never happens when one of the cone angles is greater than $2\pi$.
        Again, this does not affect our constructions and proofs.
    \end{rem}

    \begin{figure}[ht]
        \centerline{\definecolor{qqwuqq}{rgb}{0,0.39215686274509803,0}
\definecolor{zzttqq}{rgb}{0.6,0.2,0}
\definecolor{xdxdff}{rgb}{0.49019607843137253,0.49019607843137253,1}
\definecolor{ududff}{rgb}{0.30196078431372547,0.30196078431372547,1}
\begin{tikzpicture}[line cap=round,line join=round,>=triangle 45,x=1cm,y=1cm]
\clip(-9.6,-6) rectangle (1, 3.5);
\fill[line width=2pt,color=zzttqq,fill=zzttqq,fill opacity=0.10000000149011612] (-9.487611590983445,-1.2132981263894362) -- (-7.0608943862521425,0.12168082722345996) -- (-6.943317127439943,2.8888636521445923) -- (-5.32176650662542,1.0675450772689248) -- (-3.767611590983443,2.946701873610564) -- (-3.877661790646921,-0.41229671083925723) -- (-1.4076115909834437,1.8667018736105638) -- (-1.2318159730154206,-0.024458772696855036) -- (0.5923884090165564,-0.5532981263894362) -- (-1.9476241711822107,-1.9073149838557844) -- (-0.5676115909834436,-4.433298126389436) -- (-3.069803371805362,-3.114120044197656) -- (-4.087611590983443,-5.753298126389437) -- (-5.6784735016385595,-3.1953257522427654) -- (-7.527611590983443,-5.573298126389437) -- (-7.161230922854422,-2.788044431542444) -- cycle;
\draw [shift={(-7.0608943862521425,0.12168082722345996)},line width=2pt,color=qqwuqq,fill=qqwuqq,fill opacity=0.10000000149011612] (0,0) -- (87.56697283943201:0.5505281041584563) arc (87.56697283943201:208.8159504538192:0.5505281041584563) -- cycle;
\draw [shift={(-7.161230922854422,-2.788044431542444)},line width=2pt,color=qqwuqq,fill=qqwuqq,fill opacity=0.10000000149011612] (0,0) -- (145.90560397452543:0.5505281041584563) arc (145.90560397452543:262.50616611446736:0.5505281041584563) -- cycle;
\draw [shift={(-5.6784735016385595,-3.1953257522427654)},line width=2pt,color=qqwuqq,fill=qqwuqq,fill opacity=0.10000000149011612] (0,0) -- (-127.86903846883851:0.5505281041584563) arc (-127.86903846883851:-58.121568106366965:0.5505281041584563) -- cycle;
\draw [shift={(-3.069803371805362,-3.114120044197656)},line width=2pt,color=qqwuqq,fill=qqwuqq,fill opacity=0.10000000149011612] (0,0) -- (-111.08930686773265:0.5505281041584563) arc (-111.08930686773265:-27.798602693100417:0.5505281041584563) -- cycle;
\draw [shift={(-1.9476241711822107,-1.9073149838557844)},line width=2pt,color=qqwuqq,fill=qqwuqq,fill opacity=0.10000000149011612] (0,0) -- (-61.351032414707:0.5505281041584563) arc (-61.351032414707:28.060955985355122:0.5505281041584563) -- cycle;
\draw [shift={(-1.2318159730154206,-0.024458772696855036)},line width=2pt,color=qqwuqq,fill=qqwuqq,fill opacity=0.10000000149011612] (0,0) -- (-16.166946304338524:0.5505281041584563) arc (-16.166946304338524:95.31075193211706:0.5505281041584563) -- cycle;
\draw [shift={(-3.877661790646921,-0.41229671083925723)},line width=2pt,color=qqwuqq,fill=qqwuqq,fill opacity=0.10000000149011612] (0,0) -- (42.69626462705595:0.5505281041584563) arc (42.69626462705595:88.12350104055201:0.5505281041584563) -- cycle;
\draw [shift={(-5.32176650662542,1.0675450772689248)},line width=2pt,color=qqwuqq,fill=qqwuqq,fill opacity=0.10000000149011612] (0,0) -- (50.40758154123754:0.5505281041584563) arc (50.40758154123754:131.67921121127873:0.5505281041584563) -- cycle;
\draw [line width=2pt,color=zzttqq] (-9.487611590983445,-1.2132981263894362)-- (-7.0608943862521425,0.12168082722345996);
\draw [line width=2pt,color=zzttqq] (-7.0608943862521425,0.12168082722345996)-- (-6.943317127439943,2.8888636521445923);
\draw [line width=2pt,color=zzttqq] (-6.943317127439943,2.8888636521445923)-- (-5.32176650662542,1.0675450772689248);
\draw [line width=2pt,color=zzttqq] (-5.32176650662542,1.0675450772689248)-- (-3.767611590983443,2.946701873610564);
\draw [line width=2pt,color=zzttqq] (-3.767611590983443,2.946701873610564)-- (-3.877661790646921,-0.41229671083925723);
\draw [line width=2pt,color=zzttqq] (-3.877661790646921,-0.41229671083925723)-- (-1.4076115909834437,1.8667018736105638);
\draw [line width=2pt,color=zzttqq] (-1.4076115909834437,1.8667018736105638)-- (-1.2318159730154206,-0.024458772696855036);
\draw [line width=2pt,color=zzttqq] (-1.2318159730154206,-0.024458772696855036)-- (0.5923884090165564,-0.5532981263894362);
\draw [line width=2pt,color=zzttqq] (0.5923884090165564,-0.5532981263894362)-- (-1.9476241711822107,-1.9073149838557844);
\draw [line width=2pt,color=zzttqq] (-1.9476241711822107,-1.9073149838557844)-- (-0.5676115909834436,-4.433298126389436);
\draw [line width=2pt,color=zzttqq] (-0.5676115909834436,-4.433298126389436)-- (-3.069803371805362,-3.114120044197656);
\draw [line width=2pt,color=zzttqq] (-3.069803371805362,-3.114120044197656)-- (-4.087611590983443,-5.753298126389437);
\draw [line width=2pt,color=zzttqq] (-4.087611590983443,-5.753298126389437)-- (-5.6784735016385595,-3.1953257522427654);
\draw [line width=2pt,color=zzttqq] (-5.6784735016385595,-3.1953257522427654)-- (-7.527611590983443,-5.573298126389437);
\draw [line width=2pt,color=zzttqq] (-7.527611590983443,-5.573298126389437)-- (-7.161230922854422,-2.788044431542444);
\draw [line width=2pt,color=zzttqq] (-7.161230922854422,-2.788044431542444)-- (-9.487611590983445,-1.2132981263894362);
\begin{scriptsize}
\draw [fill=ududff] (-9.487611590983445,-1.2132981263894362) circle (2.5pt);
\draw[shift={(0.1,-0.3)},color=ududff] (-9.319763443723945,-1.4878347759151394) node {$P_{n-1}$};
\draw [fill=ududff] (-6.943317127439943,2.8888636521445923) circle (2.5pt);
\draw[shift={(0.1,-0.3)},color=ududff] (-6.613000264944868,3.118250362210617) node {$P_1$};
\draw [fill=ududff] (-3.767611590983443,2.946701873610564) circle (2.5pt);
\draw[shift={(0.1,-0.3)},color=ududff] (-3.4199372608258223,2.916390057352516) node {$P_2$};
\draw [fill=ududff] (-1.4076115909834437,1.8667018736105638) circle (2.5pt);
\draw[shift={(0.1,-0.3)},color=ududff] (-1.2361757809972789,2.2924582059729315) node {$P_3$};
\draw [fill=ududff] (0.5923884090165564,-0.5532981263894362) circle (2.5pt);
\draw [fill=ududff] (-0.5676115909834436,-4.433298126389436) circle (2.5pt);
\draw [fill=ududff] (-4.087611590983443,-5.753298126389437) circle (2.5pt);
\draw [fill=ududff] (-7.527611590983443,-5.573298126389437) circle (2.5pt);
\draw[shift={(0.1,-0.3)},color=ududff] (-6.989194469453146,-5.323180568219056) node {$P_{n-2}$};
\draw [fill=xdxdff] (-7.0608943862521425,0.12168082722345996) circle (2.5pt);
\draw[shift={(0.1,-0.3)},color=xdxdff] (-6.723105885776559,0.4757154622500237) node {$X_1$};
\draw [fill=xdxdff] (-5.32176650662542,1.0675450772689248) circle (2.5pt);
\draw[shift={(0.1,-0.3)},color=xdxdff] (-5.016468762885345,1.0629454400190443) node {$X_2$};
\draw [fill=xdxdff] (-3.877661790646921,-0.41229671083925723) circle (2.5pt);
\draw[shift={(0.1,-0.3)},color=xdxdff] (-3.4749900712416677,-0.3500766939876618) node {$X_3$};
\draw [fill=xdxdff] (-1.2318159730154206,-0.024458772696855036) circle (2.5pt);
\draw [fill=xdxdff] (-1.9476241711822107,-1.9073149838557844) circle (2.5pt);
\draw [fill=xdxdff] (-3.069803371805362,-3.114120044197656) circle (2.5pt);
\draw [fill=xdxdff] (-5.6784735016385595,-3.1953257522427654) circle (2.5pt);
\draw[shift={(0.1,-0.3)},color=xdxdff] (-5.374312030588341,-2.7540494154795905) node {$X_{n-2}$};
\draw [fill=xdxdff] (-7.161230922854422,-2.788044431542444) circle (2.5pt);
\draw[shift={(0.1,-0.3)},color=xdxdff] (-6.860737911816173,-2.331977868958107) node {$X_{n-1}$};
\draw[shift={(-.4,-0.3)},color=qqwuqq] (-7.4754942947931164,0.5674701462764331) node {$\alpha_1$};
\draw[shift={(-.4,-0.3)},color=qqwuqq] (-7.392915079169347,-3.0660153411693827) node {$\alpha_{n-1}$};
\draw[shift={(-.4,-0.3)},color=qqwuqq] (-5.429364841004187,-3.6899471925489675) node {$\alpha_{n-2}$};
\draw[shift={(-.4,-0.3)},color=qqwuqq] (-3.2180769559677214,0.5491192094711512) node {$\alpha_3$};
\draw[shift={(-.4,-0.3)},color=qqwuqq] (-4.9797668892747815,2.1089488379201127) node {$\alpha_2$};
\end{scriptsize}
\end{tikzpicture}}
        \caption{Polygonal coordinates.}
        \label{fig:pol}
    \end{figure}
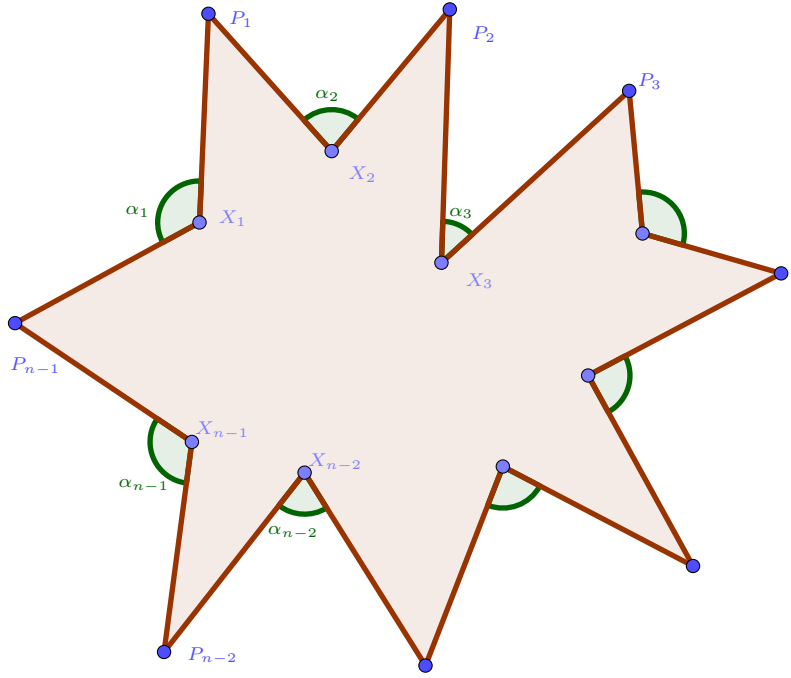

    \begin{proposition}\cite[\S 2.2]{schwartz2015notes},\ \cite[Prop.~3.2, also \S 7]{thurston1998shapes}
        The expression
        \[(g(P_1)-g(P_{n-1}), g(P_2)-g(P_{n-1}), \dots, g(P_{n-2})-g(P_{n-1}))\]
        provides independent coordinates for a developing map $g$ in the space $\mathcal{T}_\alpha$.
    \end{proposition}

    \begin{lemma}
        In these coordinates the matrix $H$ of the Hermitian form $\mathbf{A}$ is
        \[
            -\frac{1}{4}
            \begin{pmatrix}
                c_1+c_2 & c_2-i   & 0       & \dots  & 0               & 0               \\
                c_2+i   & c_2+c_3 & c_3-i   & \dots  & 0               & 0               \\
                0       & c_3+i   & c_3+c_4 & \dots  & 0               & 0               \\
                \vdots  & \vdots  & \vdots  & \ddots & \vdots          & \vdots          \\
                0       & 0       & 0       & \dots  & c_{n-3}+c_{n-2} & c_{n-2}-i       \\
                0       & 0       & 0       & \dots  & c_{n-2}+i       & c_{n-2}+c_{n-1}
            \end{pmatrix}
        \]
        where $c_i = \cot(\alpha_i / 2)$.
        Its determinant equals
        \[
            \frac{(-1)^{n-1}\sin\left(\frac{\alpha_n}{2}\right)}{4^{n-2}\sin\left(\frac{\alpha_1}{2}\right)\dots
            \sin\left(\frac{\alpha_{n-1}}{2}\right)}.
        \]
    \end{lemma}

    \begin{proof}
        The area of the polygon may be calculated as the area of bigger polygon
        $g(P_1),g(P_2),\dots,g(P_{n-1})$ minus areas of the triangles $g(P_i),g(P_{i+1}),g(X_{i+1})$.
        The former is the sum of areas of the triangles $g(P_{n-1}),g(P_i),g(P_{i+1})$ which equal
        \[
            \frac{-i}{4}\left(\left(\overline{g(P_i)}-\overline{g(P_{n-1})}\right)
            (g(P_{i+1})-g(P_{n-1})) -
            \left(\overline{g(P_{i+1})}-\overline{g(P_{n-1})}\right)(g(P_{i})-g(P_{n-1}))\right).
        \]
        And the latter ones are
        \[
            \frac{c_{i+1}}{4}\left(\overline{g(P_{i+1})}-\overline{g(P_{i})}\right)(g(P_{i+1})-g(P_{i})).
        \]
        Summing these up will give the desired matrix.
        The claim about its determinant can be proved by the induction on $n$ if we replace $\alpha_n$
        with $4\pi-\alpha_1-\dots-\alpha_{n-1}$, which does not change the desired statement.
    \end{proof}

    \section{Truncation of a metric}\label{sec:metrics_truncation}

    In this section we introduce a flow on the moduli space of flat metrics with cone singularities.

    Consider a flat cone sphere $m$ representing an element $\widetilde{m} \in \mathcal{M}_{\varphi, \alpha}$.
    Let $\gamma$ be the shortest geodesic between singular points $X_1$ and $X_2$ on sphere with metric $m$.
    Let $l$ be the length of $\gamma$.
    Let us introduce the following surgery of metric $m$ in the neighborhood of $\gamma$.

    \begin{figure}[h]
        \centerline{\definecolor{zzttqq}{rgb}{0.6,0.2,0}
\definecolor{qqwuqq}{rgb}{0,0.39215686274509803,0}
\definecolor{ududff}{rgb}{0.30196078431372547,0.30196078431372547,1}
\begin{tikzpicture}[line cap=round,line join=round,>=triangle 45,x=.3cm,y=.3cm]
\clip(-15.450759686509766,-8) rectangle (28.69208811963514,8);
\draw [shift={(-10,0)},line width=1pt,color=qqwuqq,fill=qqwuqq,fill opacity=0.10000000149011612] (0,0) -- (-123.6900675259798:1.2680293539809895) arc (-123.6900675259798:123.69006752597976:1.2680293539809895) -- cycle;
\draw [shift={(-5,0)},line width=1pt,color=qqwuqq,fill=qqwuqq,fill opacity=0.10000000149011612] (0,0) -- (71.56505117707799:1.2680293539809895) arc (71.56505117707799:288.434948822922:1.2680293539809895) -- cycle;
\fill[line width=1pt,color=zzttqq,fill=zzttqq,fill opacity=0.10000000149011612] (-14,6) -- (-10,0) -- (-5,0) -- (-3,6) -- cycle;
\fill[line width=1pt,color=zzttqq,fill=zzttqq,fill opacity=0.10000000149011612] (-14,-6) -- (-10,0) -- (-5,0) -- (-3,-6) -- cycle;
\fill[line width=1pt,color=zzttqq,fill=zzttqq,fill opacity=0.10000000149011612] (2,6) -- (3.869365983309148,3.1959510250362784) -- (12.065317008345426,3.1959510250362784) -- (13,6) -- cycle;
\fill[line width=1pt,color=zzttqq,fill=zzttqq,fill opacity=0.10000000149011612] (2,-6) -- (3.869365983309148,-3.1959510250362784) -- (12.065317008345426,-3.1959510250362784) -- (13,-6) -- cycle;
\fill[line width=1pt,color=zzttqq,fill=zzttqq,fill opacity=0.10000000149011612] (16.789864016690853,2.8040489749637216) -- (18.65923,0) -- (26.85518102503628,0) -- (27.789864016690853,2.8040489749637216) -- cycle;
\fill[line width=1pt,color=zzttqq,fill=zzttqq,fill opacity=0.10000000149011612] (16.789864016690853,-2.8040489749637216) -- (18.65923,0) -- (26.85518102503628,0) -- (27.789864016690853,-2.8040489749637216) -- cycle;
\draw [shift={(18.65923,0)},line width=1pt,color=qqwuqq,fill=qqwuqq,fill opacity=0.10000000149011612] (0,0) -- (-123.6900675259798:1.2680293539809895) arc (-123.6900675259798:123.69006752597976:1.2680293539809895) -- cycle;
\draw [shift={(26.85518102503628,0)},line width=1pt,color=qqwuqq,fill=qqwuqq,fill opacity=0.10000000149011612] (0,0) -- (71.56505117707802:1.2680293539809895) arc (71.56505117707802:288.434948822922:1.2680293539809895) -- cycle;
\draw [->,line width=1pt] (-6,20) -- (10,20);
\draw [line width=1pt] (6,0)-- (11,0);
\draw [line width=1pt] (3.869365983309148,3.1959510250362784)-- (6,0);
\draw [line width=1pt] (6,0)-- (3.869365983309148,-3.1959510250362784);
\draw [line width=1pt] (11,0)-- (12.065317008345426,3.1959510250362784);
\draw [line width=1pt] (11,0)-- (12.065317008345426,-3.1959510250362784);
\draw [line width=1pt,dotted] (-10,0)-- (-14,6);
\draw [line width=1pt] (-10,0)-- (-5,0);
\draw [line width=1pt,dotted] (-5,0)-- (-3,6);
\draw [line width=1pt,dotted] (-10,0)-- (-14,-6);
\draw [line width=1pt,dotted] (-3,-6)-- (-5,0);
\draw [line width=1pt,dotted] (3.869365983309148,3.1959510250362784)-- (2,6);
\draw [line width=1pt,dotted] (12.065317008345426,3.1959510250362784)-- (13,6);
\draw [line width=1pt] (3.869365983309148,3.1959510250362784)-- (12.065317008345426,3.1959510250362784);
\draw [line width=1pt,dotted] (3.869365983309148,-3.1959510250362784)-- (2,-6);
\draw [line width=1pt] (3.869365983309148,-3.1959510250362784)-- (12.065317008345426,-3.1959510250362784);
\draw [line width=1pt,dotted] (12.065317008345426,-3.1959510250362784)-- (13,-6);
\draw [line width=1pt,dotted] (16.789864016690853,-2.8040489749637216)-- (18.65923,0);
\draw [line width=1pt,dotted] (18.65923,0)-- (16.789864016690853,2.8040489749637216);
\draw [line width=1pt,dotted] (27.789864016690853,2.8040489749637216)-- (26.85518102503628,0);
\draw [line width=1pt,dotted] (26.85518102503628,0)-- (27.789864016690853,-2.8040489749637216);
\draw [line width=1pt] (18.65923,0)-- (26.85518102503628,0);
\draw [<-, shift={(15.571164908021382,-1.8467791580114366)},line width=1pt]  plot[domain=0.43284915475094893:2.3590740561315817,variable=\t]({1*9.17546655123954*cos(\t r)+0*9.17546655123954*sin(\t r)},{0*9.17546655123954*cos(\t r)+1*9.17546655123954*sin(\t r)});
\draw [->, shift={(15.571164908021382,1.8467791580114366)},line width=1pt]  plot[domain=3.9241112510480045:5.850336152428637,variable=\t]({1*9.17546655123954*cos(\t r)+0*9.17546655123954*sin(\t r)},{0*9.17546655123954*cos(\t r)+1*9.17546655123954*sin(\t r)});
\draw [->, line width=1pt] (-2,0)-- (2,0);
\draw [{Stealth}-{Stealth},line width=.5pt,dash pattern=on 0.5pt off 1pt] (9.19127742585413,3.1959510250362784)-- (9.19127742585413,0);
\draw [{Stealth}-{Stealth},line width=.5pt,dash pattern=on 0.5pt off 1pt] (9.19127742585413,0)-- (9.19127742585413,-3.195951025036279);
\begin{scriptsize}
\draw [fill=ududff] (-10,0) circle (1.5pt);
\draw [fill=ududff] (-5,0) circle (1.5pt);
\draw[shift={(-1.2,-0.2)},color=qqwuqq] (-7.80031591749113,1.980818043396061) node {$\varphi_1$};
\draw[shift={(-1.2,-0.2)},color=qqwuqq] (-4.334369016609758,2.0230856885287603) node {$\varphi_2$};
\draw [fill=ududff] (-6,20) circle (1.5pt);
\draw [fill=ududff] (10,20) circle (1.5pt);
\draw [fill=ududff] (6,0) circle (1.5pt);
\draw [fill=ududff] (11,0) circle (1.5pt);
\draw [fill=ududff] (3.869365983309148,3.1959510250362784) circle (1.5pt);
\draw [fill=ududff] (12.065317008345426,3.1959510250362784) circle (1.5pt);
\draw [fill=ududff] (3.869365983309148,-3.1959510250362784) circle (1.5pt);
\draw [fill=ududff] (12.065317008345426,-3.1959510250362784) circle (1.5pt);
\draw [fill=ududff] (18.65923,0) circle (1.5pt);
\draw [fill=ududff] (26.85518102503628,0) circle (1.5pt);
\draw[shift={(-1.2,-0.2)},color=black] (-7.250836530766033,0.9241269150785709) node {$\gamma$};
\draw[shift={(-1.2,-0.2)},color=qqwuqq] (21.617965094867824,1.4736063018036658) node {$\varphi_1$};
\draw[shift={(-1.2,-0.2)},color=qqwuqq] (26.901420736455282,1.6004092372017644) node {$\varphi_2$};
\draw[shift={(-1.2,-0.2)},color=black] (9.825292102844625,1.9385503982633614) node {r};
\draw[shift={(-1.2,-0.2)},color=black] (9.825292102844625,-1.10472005129101) node {r};
\draw [fill=ududff] (5.869365983309148,3.2359510250362784) node {\small \ScissorRightBrokenBottom};
\draw [fill=ududff] (5.869365983309148,-3.1459510250362784) node {\small \ScissorRightBrokenBottom};
\end{scriptsize}
\end{tikzpicture}}
        \caption{Metric truncation}
        \label{fig:tr}
    \end{figure}
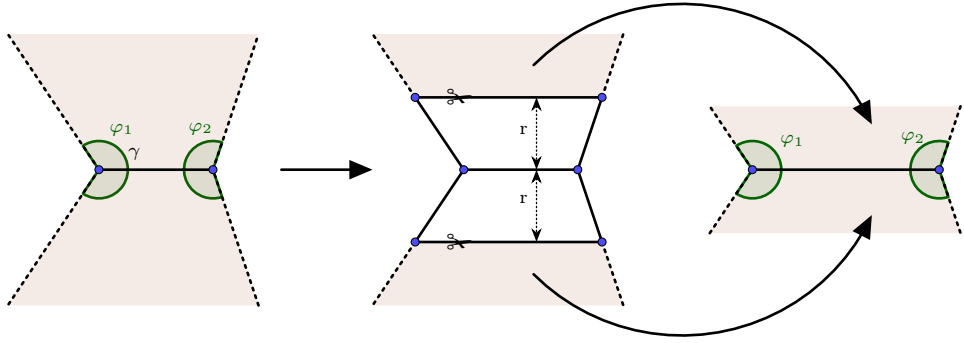

    Some neighborhood of $\gamma$ can be represented as a flat domain obtained from the plane by cutting out
    two sectors with vertices at points $X_1$ and $X_2$ and gluing the sides (see Fig.~\ref{fig:tr}, left).
    Suppose the cut-out sectors are symmetric with respect to line $X_1 X_2$, which means that the segment lies
    on the bisector of both green angles on Fig.~\ref{fig:tr}, left.
    The geodesic $\gamma$ itself is presented as a segment $X_1 X_2$.
    Then we can consider two lines parallel to this segment in equal distance $r$ from it (Fig.~\ref{fig:tr}, middle).
    For small enough $r$ these two lines bound a ``two-angled'' area on the sphere.
    By cutting out this area and by gluing sides we obtain a new flat cone sphere with the
    same list of cone angles (Fig.~\ref{fig:tr}, right).

    There exists a vector field $v$ on $\mathcal{M}_{\varphi, \alpha}$ such that its flow will correspond to metric truncation.
    Precisely, for an element of $\mathcal{M}_{\varphi, \alpha}$ consider a metric $m=m_0$ with length $l_0=l(m_0)$ of the
    shortest geodesic and for small $t$ define $m_t$ as truncation of $m$ by width $t\,l_0$.
    From this we can immediately get the derivative of the length $l(m_t)$ and of the area $a(m_t)$ with respect to $t$:
    \begin{align*}
        \frac{\partial l(m_t)}{\partial t}|_{t=0} &= -t\,l_0\,\left(\cot \frac{\varphi_1}{2} +
        \cot\frac{\varphi_2}{2}\right),\\
        \frac{\partial a(m_t)}{\partial t}|_{t=0}  &= -2t\,l_0^2.
    \end{align*}
    The family $\widetilde{m}_t\in \mathcal{M}_{\varphi,\alpha}$ corresponding to metrics $m_t$ defines a tangent vector
    $v = \frac{\partial\widetilde{m}_t}{\partial t}$ at the point $\widetilde{m}_0\in \mathcal{M}_{\varphi, \alpha}$.

    This construction gives a smooth vector field on $\mathcal{M}_{\varphi, \alpha}$ which is well-defined whenever
    the shortest geodesic between $X_1$ and $X_2$ is unique.
    Denote by $\Omega \subset \mathcal{M}_\alpha$ the complementary set, i.e.\ the set of those metrics where there are
    at leat two shortest geodesics between $X_1$ and $X_2$.
    The structure of $\Omega$ will be explained in the following section, but it is clear now that
    it will be some union of submanifolds of positive codimensions, so the vector field $v$
    is well-defined on an open dense subset of $\mathcal{M}_\alpha$.

    \section{Multiple shortest geodesics}\label{sec:multiple_shortest}

    In this section we describe the set $\Omega$ of elements from $\mathcal{M}_{\varphi,\alpha}$ having at least two
    shortest geodesics between the distinguished points.

    Consider some flat cone sphere $m$ with two shortest geodesics $\gamma_1, \gamma_2$ between
    the distinguished singularities $X_1, X_2$.
    Clearly, it means that they are of equal length and do not have common points except $X_1, X_2$.

    Let us cut our sphere along $\gamma_1, \gamma_2$.
    The sphere splits into two connected components.
    By gluing the sides of each of them we obtain a pair of flat cone spheres $\widehat{m}, \widetilde{m}$.
    The distinguished singularities with angles $\varphi_1, \varphi_2$ of $m$ split into singularities
    with angles $\widehat{\varphi}_1, \widetilde{\varphi}_1,\widehat{\varphi}_2, \widetilde{\varphi}_2$ of which
    $\widehat{\varphi}_1, \widehat{\varphi}_2$ are angles of singularities in $\widehat{m}$ and
    $\widetilde{\varphi}_1, \widetilde{\varphi}_2$ are angles of singularities in $\widetilde{m}$.
    Each of the non-distinguished singularities in $m$ goes either to $\widehat{m}$ or to $\widetilde{m}$.
    Therefore $\widehat{m}, \widetilde{m}$ represent some classes in $\mathcal{M}_{\widehat{\varphi},\widehat{\alpha}}$ and
    $\mathcal{M}_{\widetilde{\varphi},\widetilde{\alpha}}$ respectively with $\widehat{\varphi},\widetilde{\varphi},\widehat{\alpha},\widetilde{\alpha}$ satisfying
    \begin{align*}
        &\varphi_1 = \widehat{\varphi}_1 + \widetilde{\varphi}_1;\\
        &\varphi_2 = \widehat{\varphi}_2 + \widetilde{\varphi}_2;\\
        &\alpha \text{ is the union of } \widehat{\alpha} \text{ and } \widetilde{\alpha}.
    \end{align*}

    \begin{figure}[ht]
        \centerline{\input{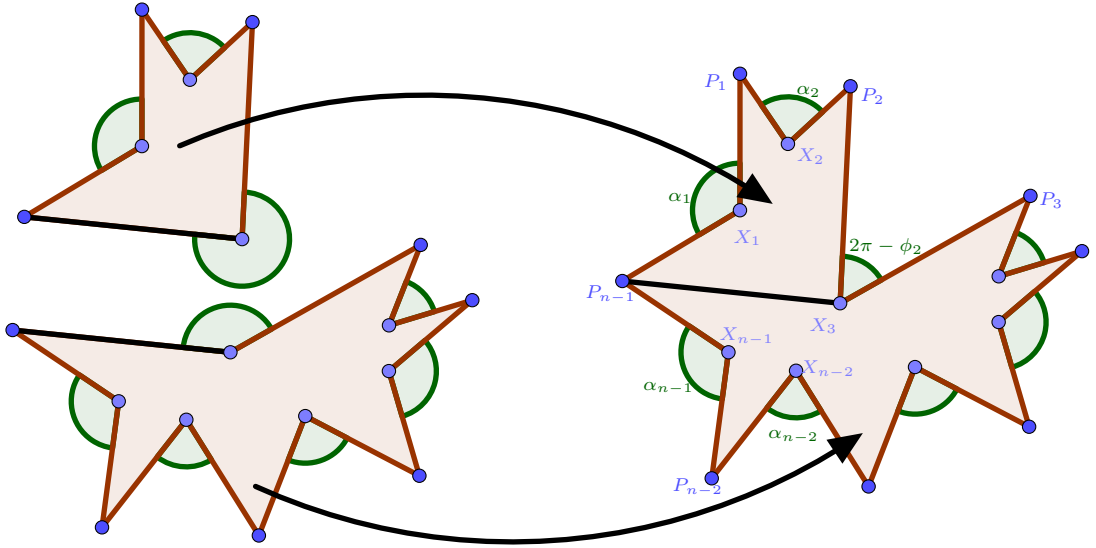}}
        \caption{Gluing of the polygons}
        \label{fig:w}
    \end{figure}

    On the other side, the original flat cone sphere $m$ is completely determined by flat cone spheres
    $m_1, m_2$, provided that the latter ones have unique shortest geodesic between the two marked singular points.
    Moreover, for any pair of elements from $\mathcal{M}_{\widehat{\varphi},\widehat{\alpha}}$ and 
    $\mathcal{M}_{\widetilde{\varphi},\widetilde{\alpha}}$
    one can consider the following procedure.

    \begin{enumerate}
        \item Choose flat cone spheres $m_1$ and $m_2$ representing some elements from
            $\mathcal{M}_{\widehat{\varphi},\widehat{\alpha}}$ and $\mathcal{M}_{\widetilde{\varphi},\widetilde{\alpha}}$
            such that the shortest geodesics between the marked points in $m_1$ and $m_2$ are of equal length.
        \item Cut $m_1$ and $m_2$ along these shortest geodesics.
        \item Glue the resulting discs to each other such that the corresponding marked points stick together.
            There is the only way to do it without change of orientations of the disks.
        \item For resulting flat cone sphere $m$ consider its class in $\mathcal{M}_{\varphi, \alpha}$.
    \end{enumerate}

    This construction gives us a partial map from 
    $\mathcal{M}_{\widehat{\varphi},\widehat{\alpha}}\times \mathcal{M}_{\widetilde{\varphi},\widetilde{\alpha}}$ to $\Omega$.
    This map is well-defined when the corresponding flat cone spheres have unique shortest geodesic between marked singular points.
    Denote sets of such points by $M'_{\widehat{\varphi},\widehat{\alpha}},M'_{\widetilde{\varphi},\widetilde{\alpha}}$.
    Obviously, they are open dense subsets of corresponding moduli spaces.
    If we restrict our map to this set we will have an embedding because our construction is the inverse to splitting construction.
    Therefore, we have proved the following.

    \begin{lemma}\label{omega_structure}
        The set $\Omega\subset \mathcal{M}_{\varphi,\alpha}$ can be represented as
        $\Omega = \left(\bigcup_\PC \Omega_\PC\right) \cup \Omega_0$ where $\Omega_0$ is
        a set of codimension at least two in $\mathcal{M}_\alpha$ and the union is over shuffles $\PC$ of $\alpha$ into
        two lists $\widehat{\alpha}, \widetilde{\alpha}$ and sets $\Omega_\PC$ have the following structure.

        There is a map $p$ from $\Omega_\PC$ to $\mathbb{R}_+$ such that $p^{-1}(\beta)$ is in
        one-to-one correspondence with $M'_{\widehat{\varphi}(\beta), \widehat{\alpha}}\times M'_{\widetilde{\varphi}(\beta),\widetilde{\alpha}}$
        or is empty if one of these moduli spaces is empty.
    \end{lemma}

    In polygonal coordinates this splitting-gluing construction looks like at Figure~\ref{fig:w},
    two polygons are glued by sides of equal length.
    Left side shows two polygons corresponding to some flat cone sphere, right side shows the result of the gluing.
    The angles marked in the figure correspond to angle defects of the flat cone spheres.
    Note that during the gluing procedure the area functions add up, i.e.
    $\mathbf{a}(m) = \mathbf{a}(m_1) + \mathbf{a}(m_2)$.

    \section{Proof of Theorem 1}\label{sec:proof_t1}
    Theorem 1 can be proved using the Stokes' theorem for the form $\iota_{v}\mathrm{Vol}$ on $\mathcal{M}_{\varphi, \alpha}$.
    Recall that on the stratified submanifold $\Omega$ the vector field $v$, and thus
    the form $\iota_{v}\mathrm{Vol}$, has a discontinuity.
    Denote by $\mathcal{M}_{\varphi,\alpha}^\Omega$ the result of cutting $\mathcal{M}_{\varphi,\alpha}$ by $\Omega$.

    \begin{lemma}
        \label{st}
        Let $g$ be a smooth function on $\mathbb{R}_+$ with compact support.
        Then
        \begin{align}
            \int_{\mathcal{M}_{\varphi,\alpha}}d\iota_{g(\mathbf{a})v}\mathrm{Vol}=
            \int_{\partial (\mathcal{M}_{\varphi,\alpha}^\Omega)}\iota_{g(\mathbf{a})v}\mathrm{Vol}
        \end{align}
    \end{lemma}

    \begin{proof}
        By Thurston's theorem, the metric completion $\overline{M}_{\varphi,\alpha}$ of the moduli space is a complex hyperbolic cone manifold, so its singular set is a union of strata of real codimension at least two.

        Let us cut out $\epsilon$-neighborhood of the singular set from $\overline{M}_{\varphi,\alpha}$ and cut the result by $\Omega$.

        The resulting space $\mathcal{M}_{\varphi,\alpha}^\Omega(\epsilon)$ is a compact manifold with boundary.
        So there holds the Stokes theorem for the form $\iota_{g(\mathbf{a})v}\mathrm{Vol}$, namely
        \begin{align}
            \label{eps-stokes}\int_{\mathcal{M}_{\varphi,\alpha}^\Omega(\epsilon)}d\iota_{g(\mathbf{a})v}\mathrm{Vol}=\int_{\partial \mathcal{M}_{\varphi,\alpha}^\Omega(\epsilon)}\iota_{g(\mathbf{a})v}\mathrm{Vol}.
        \end{align}

        Note that the vector field $v$ as well as its divergence is bounded in polygonal coordinates.
        Then its integral over $\epsilon$-neighborhood of singular set is bounded by the volume of this neighborhood multiplied by some constant.
        Therefore, it tends to zero as $\epsilon\to 0$.

        Also, the flow of $v$ through the boundary of the $\epsilon$-neighborhood tends to zero because the area of this boundary tends to zero.

        And the part of $\partial(\mathcal{M}_{\varphi,\alpha}^\Omega(\epsilon))$ coming from points of $\Omega$ is $\partial(\mathcal{M}_{\varphi,\alpha}^\Omega)$ without the area that tends to zero.
        Therefore, as the limit of equation~\ref{eps-stokes} we have our Lemma.

    \end{proof}

    The following two lemmas will be proved in the subsequent sections.

    \begin{lemma}
        \label{div}
        For $g$ as above the following holds:
        \[
            d\iota_{g(\mathbf{a})v}\mathrm{Vol} =
            2\left(\left(\frac{n}{\mathbf{a}} + \cot\frac{\varphi_1}{2}+\cot\frac{\varphi_2}{2}\right)g(\mathbf{a}) +
            \left(\left(\cot\frac{\varphi_1}{2}+\cot\frac{\varphi_2}{2}\right)\mathbf{a} - 1\right)g'(\mathbf{a})\right)\mathrm{Vol}.
        \]
    \end{lemma}

    \begin{lemma}\label{integrand}
        Let $\PC=(\widehat{\alpha},\widetilde{\alpha})$ be a shuffle of $\alpha$.
        Using the function $\beta$ on $\Omega_\PC$ with level set
        $M'_{\widehat{\varphi}(\beta), \widehat{\alpha}}\times M'_{\widetilde{\varphi}(\beta),\widetilde{\alpha}}$
        from Lemma~\ref{omega_structure}, the integrand over $\Omega_\PC$ in the right hand side of the formula in Lemma~\ref{st} is
        \[
            2\frac{\mathbf{a}_1^{n_1}\mathbf{a}_2^{n_2}}{\mathbf{a}^{n}}
            g(\mathbf{a})\,
            \varkappa(\PC,\beta)\,
            \mathrm{d}\beta\wedge \mathrm{Vol}_1\wedge \mathrm{Vol}_2
        \]
        where $\varkappa = \cot\frac{\beta}{2} +\cot\frac{\varphi_2-\beta}{2} +
                           \cot\frac{\Sigma\widehat{\alpha}-\beta}{2} +\cot\frac{\varphi_1-\Sigma\widehat{\alpha}+\beta}{2}$
        and $\mathrm{Vol}_{1,2}$ are the volume forms on $\mathcal{M}_{\widehat{\varphi}(\beta),\widehat{\alpha}}, \mathcal{M}_{\widetilde{\varphi}(\beta),\widetilde{\alpha}}$ and $\mathbf{a}_{1,2}$ are the area functions there.
    \end{lemma}

    By Lemma~\ref{div} and Lemma~\ref{integrand} we obtain

    \begin{align*}
        \int_{\mathcal{M}_{(\varphi, \alpha)}}
            2\left(\left(\frac{n}{\mathbf{a}} + \cot\frac{\varphi_1}{2}+\cot\frac{\varphi_2}{2}\right)g(\mathbf{a}) +
            \left(\left(\cot\frac{\varphi_1}{2}+\cot\frac{\varphi_2}{2}\right)\mathbf{a} - 1\right)g'(\mathbf{a})\right)
            \mathrm{Vol}
        = \\ =
        \sum_{(\widehat{\alpha}, \widetilde{\alpha})}
            \int_{\RR}
            \int_{\mathcal{M}_{\widehat{\varphi}(\beta),\widehat{\alpha}}\times \mathcal{M}_{\widetilde{\varphi}(\beta),\widetilde{\alpha}}}
                2\frac{\mathbf{a}_1^{n_1}\mathbf{a}_2^{n_2}}{(\mathbf{a}_1+\mathbf{a}_2)^{n}}
                g(\mathbf{a})\,
                \varkappa(\widehat{\alpha}, \widetilde{\alpha},\beta)\,
                \mathrm{d}\beta\wedge \mathrm{Vol}_1\wedge \mathrm{Vol}_2
    \end{align*}
    where
    $\varkappa = \cot\frac{\beta}{2} +\cot\frac{\varphi_2-\beta}{2} +
    \cot\frac{\widehat{\beta}}{2} +\cot\frac{\varphi_1-\widehat{\beta}}{2}$
    and $\mathbf{a}_1,\mathbf{a}_2$ are $\mathbf{a}$-functions on
    $\mathcal{M}_{\widehat{\varphi}(\beta),\widehat{\alpha}},
    \mathcal{M}_{\widetilde{\varphi}(\beta),\widetilde{\alpha}}$.

    From that by integration over level sets of $\mathbf{a}$-function we obtain the following proposition which implies Theorem~\ref{thm1}.
    \begin{proposition}
        \label{g_integrate}
        \begin{align*}
            \int _{\mathbb{R}_+}
                \left(\left(\frac{n}{a} + \cot\frac{\varphi_1}{2}+\cot\frac{\varphi_2}{2}\right)g(a) +
                \left(\left(\cot\frac{\varphi_1}{2}+\cot\frac{\varphi_2}{2}\right)a - 1\right)g'(a)\right)
                f_{\varphi,\alpha}(a)\,\mathrm{d}a
            = \\ =
            \sum_\PC \int_{\RR}\int_{\RR_+\times \RR_+}
                \frac{a_1^{n_1}a_2^{n_2}}{(a_1+a_2)^{n}}
                g(a_1+a_2)\,
                \varkappa(\PC,\beta)\,
                f(\widehat{\varphi}(\beta),\widehat{\alpha},a_1)
                f(\widetilde{\varphi}(\beta),\widetilde{\alpha},a_2)
                \mathrm{d}\beta\, \mathrm{d}a_1\, \mathrm{d}a_2
        \end{align*}
        where
        $f_{\varphi,\alpha}(a),
        f(\widehat{\varphi}(\beta),\widehat{\alpha},a_1),
        f(\widetilde{\varphi}(\beta),\widetilde{\alpha},a_2)$
        are the density functions of $\mathbf{a}$-functions on the corresponding moduli spaces.
    \end{proposition}

    \section{Divergence of the vector field}\label{sec:divergence}

    \begin{proof}[Proof of the Lemma~\ref{div}.]
        For metrics from $\mathcal{M}_{\varphi, \alpha}$ with $\varphi=(\varphi_1,\varphi_2), \alpha=(\alpha_1,\dots,\alpha_n)$
        we denote the distinguished singular points by $X_1, X_2$ and other singular points by $Y_1,\dots,Y_n$.
        Recall the construction of polygonal coordinates from section~\ref{sec:polygonal_coordinates}.
        We pass to the polygonal coordinates $y_1=P_1-P_0, \dots, y_{n}=P_n-P_0$ depicted in Figure~\ref{fig:t}.
        They correspond to a collection of paths from point $Y_n$ to all other singular points such that the
        shortest geodesic $X_1 X_2$ together with paths $Y_n X_1$ and $Y_n X_2$ constitute a
        (possibly curved, see remark~\ref{rem:noconfuse}) triangle disjoint from other paths.
        Together with these affine coordinates we will use the corresponding homogeneous coordinates $w_1 = y_1/{y_n}, \dots, w_{n-1} = y_{n-1}/{y_n}$.

        The last coordinate will keep unchanged, and we can without loss of generality say that $y_n = 1$ and $w_i = y_i$ for $1\le i\le n-1$.
        The volume form in these coordinates is
        \begin{equation}\label{eq:Vol_equation}
            \mathrm{Vol} =
            \frac
            {\det H}
            {(2i)^{n-1} \left((\overline{w}, 1) H (w, 1)^T\right) ^ {n}}
            dw_1 \wedge d\overline{w}_1 \ldots dw_{n-1} \wedge d\overline{w}_{n-1}
        \end{equation}
        where $(w, 1)$ means $(w_1, \dots, w_{n-1}, 1)$.
        The expression $(\overline{w}, 1) H (w, 1)^T$ will be referred as $\mathbf{A}$.
        The length of the shortest geodesic $X_1 X_2$ is
        $l \coloneqq |z_1 - z_2| = \left|\frac{y_1}{1-e^{i\varphi_1}}-\frac{y_1-e^{-i\varphi_2}y_2}{1-e^{-i\varphi_2}}\right|$.
        Then the $\mathbf{a}$-function on the moduli space is written in our coordinates as $\mathbf{a} = \mathbf{A}/{l^2}$.

        Since the vector field $v$ represents the metric truncation in the neighborhood of segment $X_1 X_2$,
        the flow can be considered to keep paths form $Y_n$ to other singular points except $X_1, X_2$,
        therefore points $P_0,P_2, P_3, \dots, P_n$ in plane can be considered constant.
        The points $Z_1, Z_2, P_1$ on the contrary are changing such that the velocity
        $\frac{\partial P_1}{\partial t}$ is orthogonal to $Z_1 Z_2$ (see section~\ref{sec:metrics_truncation}).
        Precisely,
        \[
            \frac{\partial w_1}{\partial t}
            =\frac{\partial P_1}{\partial t}=2i(z_1-z_2)
            =2i\left(\frac{y_1}{1-e^{i\varphi_1}}-\frac{y_1-e^{-i\varphi_2}y_2}{1-e^{-i\varphi_2}}\right).
        \]
        Therefore, the vector field $v$ can be written in the form of differential operator as
        \begin{equation}\label{eq:v_equation}
            \partial_v = v_1\frac{\partial}{\partial w_1} + \overline{v}_1\frac{\partial}{\partial \overline{w}_1}
            \text{ where } v_1 = 2i\left(
            \frac{w_1}{1-e^{i\varphi_1}} -
            \frac{w_1-e^{-i\varphi_2}w_2}{1-e^{-i\varphi_2}}
            \right)
        \end{equation}

        \begin{figure}[ht]
            \centerline{\includegraphics[width=10cm]{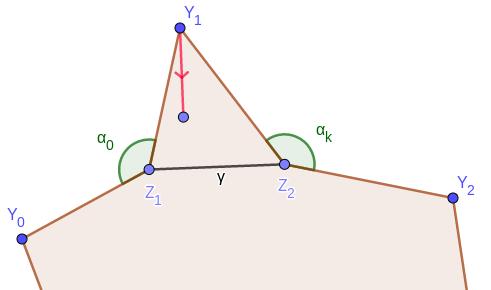}}
            \caption{Polygonal coordinates and shortest geodesic.}
            \label{fig:t}
        \end{figure}

        From formulae~\ref{eq:Vol_equation} and~\ref{eq:v_equation} we get
        \begin{align*}
            d\iota_{g(\mathbf{a})v}\mathrm{Vol} =
            \left(
                \frac{\partial}{\partial w_1}
                \frac{g(\mathbf{a}) v_1\det H}
                    {(2i)^{n-1}\mathbf{A}^{n}} +
                \frac{\partial}{\partial \overline{w}_1}
                    \frac{g(\mathbf{a}) \overline{v}_1\det H}
                {(2i)^{n-1}\mathbf{A}^{n}}
            \right)
            dw_1 \wedge d\overline{w}_1 \ldots dw_{n-3} \wedge d\overline{w}_{n-3}
            = \\ =
            \left(
                2g(\mathbf{a})\Re(\frac{\partial v_1}{\partial w_1}) +
                g'(\mathbf{a})\partial_v \mathbf{a} -
                n\,g(\mathbf{a})\frac{\partial_v \mathbf{A}}{\mathbf{A}}
            \right)\mathrm{Vol}
        \end{align*}

        A simple calculation shows that
        \[\partial_v l = -\left(\cot\frac{\varphi_1}{2}+\cot \frac{\varphi_2}{2}\right)l \quad\text{and}\quad
        \partial_v \mathbf{A} = -2l^2.\]
        Then we can derive
        \[\frac{\partial_v \mathbf{A}}{\mathbf{A}} = \frac{-2}{\mathbf{a}} \quad\text{and}\quad
        \partial_v \mathbf{a} =  \partial_v \frac{\mathbf{A}}{l^2} =
        \frac{l \partial_v\,\mathbf{A} - 2\,\mathbf{A}\,\partial_v l}{l^3} =
        -2 + 2\left(\cot\frac{\varphi_1}{2}+\cot \frac{\varphi_2}{2}\right)\mathbf{a}.\]
        Also $\Re\left(2\frac{\partial v_1}{\partial w_1}\right) =
        2\Im\left(\frac{1}{1-e^{i\varphi_1}}-\frac{1}{1-e^{-i\varphi_2}}\right) =
        \cot\frac{\varphi_1}{2}+\cot\frac{\varphi_2}{2}$.
        Thus, we have
        \[
            d\iota_{g(\mathbf{a})v}\mathrm{Vol} =
            2\left(\left(\frac{n}{\mathbf{a}} + 2\left(\cot\frac{\varphi_1}{2}+\cot\frac{\varphi_2}{2}\right)\right)g(\mathbf{a}) +
            2\left(\left(\cot\frac{\varphi_1}{2}+\cot\frac{\varphi_2}{2}\right)\mathbf{a} - 1\right)g'(\mathbf{a})\right)\mathrm{Vol}. \qedhere
        \]
    \end{proof}

    \section{Flow of the vector field through \texorpdfstring{$\Omega$}{Omega}}\label{sec:flow}

    Let $\PC=(\widehat{\alpha},\widetilde{\alpha})$ be a shuffle of $\alpha$.
    At every point of $\Omega_\PC$ the vector field $v$ has a discontinuity.
    That is due to different choices 
    of the shortest geodesic for the construction of $v$.
    For the generic point of $\Omega_\PC$ there are two such geodesics, so locally in some neighborhood $U$ there are
    two smooth vector fields $v_{1,2}$, one for each side $U_{1,2}$ of $U\setminus \Omega_\PC$.
    We need to find the flow of $v_{1,2}$ through $\Omega_\PC$ (cooriented with the normal vector outside $U_1$
    for $v_1$ and outside $U_2$ for $v_2$).
    Note that both $v_{1,2}$ are analytic and therefore can be uniquely continued to the whole $U$.

    Recall the map $p$ from $\Omega_\PC$ to $(0, 2\pi)$ from Lemma~\ref{omega_structure}.
    For all $\beta$ we have $p^{-1}(\beta)$ is
    $M'_{\widehat{\varphi}(\beta), \widehat{\alpha}}\times M'_{\widetilde{\varphi}(\beta), \widetilde{\alpha}}$.
    We note that locally the homogeneous coordinates on
    $\mathcal{M}_{\widehat{\varphi}(\beta), \widehat{\alpha}}, \mathcal{M}_{\widetilde{\varphi}(\beta), \widetilde{\alpha}}$
    for some $\beta$ can be extended to functions on some neighborhood of every point of $\Omega_\PC$.
    Moreover, that can be done so that for all $\beta'$ close to $\beta$ the restriction of these functions
    to $p^{-1}(\beta')$ will provide local coordinates on
    $\mathcal{M}_{\widehat{\varphi}(\beta), \widehat{\alpha}}, \mathcal{M}_{\widetilde{\varphi}(\beta), \widetilde{\alpha}}$.
    Therefore, the volume forms $\mathrm{Vol}_1, \mathrm{Vol}_2$ on
    $\mathcal{M}_{\widehat{\varphi}(\beta), \widehat{\alpha}}, \mathcal{M}_{\widetilde{\varphi}(\beta), \widetilde{\alpha}}$
    for different $\beta$ can be considered as well-defined forms on $\Omega_\PC$, which we also
    denote by $\mathrm{Vol}_{1,2}$ (of course, they are not volume forms there).

    Denote by $\mathbf{a_{1,2}}$ the area functions on
    $\mathcal{M}_{\widehat{\varphi}(\beta), \widehat{\alpha}}, \mathcal{M}_{\widetilde{\varphi}(\beta), \widetilde{\alpha}}$
    and also their pullbacks on $\Omega_\PC$.

    \begin{lemma}
        \label{iv}
        Choose a neighborhood $U$ in $\mathcal{M}_{\varphi, \alpha}$ of some point of $\Omega_\PC$.
        Let $v_1$ be the analytic vector field defined in $U$ and coinciding with $v$ on the side $U_1\subset U\setminus\Omega_\PC$.
        Then
        \[
            \iota_{g(\mathbf{a})v_1}\mathrm{Vol}\mid_{\Omega_\PC} =
            \frac{\mathbf{a}_\mathbf{1}^{n_1}\mathbf{a}_\mathbf{2}^{n_2}}{\mathbf{a}^{n}}
            g(\mathbf{a})\,
            \varkappa(\PC,\beta)\,
            d\beta\wedge \mathrm{Vol}_1\wedge \mathrm{Vol}_2
        \]
        where $\varkappa = \left(\cot\frac{\beta}{2} +\cot\frac{\varphi_2-\beta}{2} +\cot\frac{\widehat{\beta}}{2} +\cot\frac{\varphi_1-\widehat{\beta}}{2}\right)$
        and $\widehat{\beta} = \alpha_1+\dots+\alpha_{n_1} - \beta$.
    \end{lemma}

    \begin{proof}
        Choose a generic point $\widehat{m}\in \Omega_\PC$ and a metric $m\in\widehat{m}$ such that
        the length of shortest geodesic (both lengths of shortest geodesics) between
        distinguished singular points is unit.
        We pass to coordinates $X_1, \dots, X_{n-1}$ depicted in Figure~\ref{fig:wpc}
        (here we have changed notation and the point $X_{n_1+1}$ corresponds to cone angle $\varphi_2$).
        The base coordinate for homogeneous coordinates will be $X_{n_1}$ and these coordinates will be
        denoted as $y_1 = X_1/X_{n_1}, \dots, y_{n_1-1}=X_{n_1-1}/X_{n_1}, q = X_{n_1+1}/X_{n_1},
        \widetilde{z}_1 = X_{n_1+2}/X_{n_1}, \dots, \widetilde{z}_{n_2-1}=X_{n-1}/X_{n_1}$.
        Denote also $z_i = \widetilde{z}_i/q$ for $i=1,\dots,n_2-1$.
        Denote $r = \frac{q - e^{i\varphi_2}}{1 - e^{i\varphi_2}}$.
        Then $\Omega$ is locally the set of metrics where $r$ is the circumcenter of triangle $(0,1,q)$ and can be defined by equation $|r| = |1-r| \Leftrightarrow |\frac{r-1}{r}|=1$.

        We treat $(y_1, \dots, y_{n_1-1})$ as coordinates on
        $\mathcal{M}_{\widehat{\varphi}(\beta), \widehat{\alpha}}$ for different $\beta'$.
        And also $(z_1, \dots, z_{n_2-1})$ can be treated as coordinates on
        $\mathcal{M}_{\widetilde{\varphi}(\beta), \widetilde{\alpha}}$.

        \begin{figure}[ht]
            \centerline{\includegraphics[width=10cm]{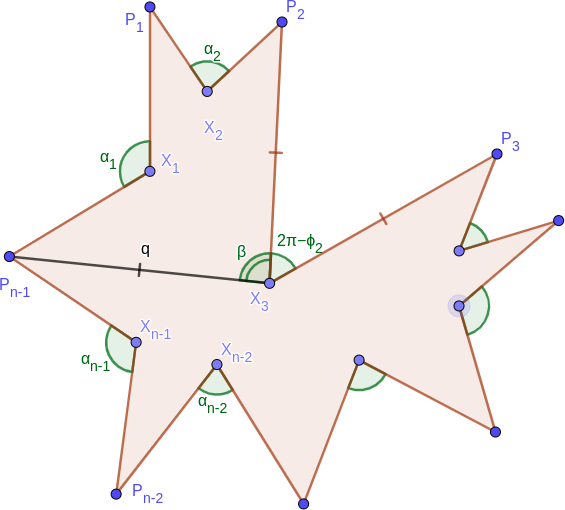}}
            \caption{Coordinates for metrics from $\Omega_\PC$.}
            \label{fig:wpc}
        \end{figure}

        The function $\beta$ from Lemma~\ref{omega_structure} on $\Omega$ can be locally written as
        $\operatorname{Arg}\left(\frac{r-1}{r}\right)$ and can be extended by this formula to some neighborhood in $\mathcal{M}_{\varphi, \alpha}$.
        Also, we will need the function $\rho=|\frac{r-1}{r}|^2$.


        First, we pass from coordinates $(y_1,\dots,y_{n_1-1},q,\widetilde{z}_1,\dots,\widetilde{z}_{n_2-1})$ to coordinates $(y_1,\dots,y_{n_1-1}$, $z_1,\dots,z_{n_2-1}, \beta, \rho)$ of which the last two are real and other are complex.
        For simplicity, we will use notation
        \begin{align*}
            DY &= \dd y_1\wedge\dd\overline{y}_1\wedge\dots\wedge\dd y_{n_1-1}\wedge\dd\overline{y}_{n_1-1}, \\
            DZ &= \dd z_1\wedge\dd\overline{z}_1\wedge\dots\wedge\dd z_{n_2-1}\wedge\dd\overline{z}_{n_2-1}.
        \end{align*}
        A direct computation shows that
        $\dd q\wedge \dd \overline{q} = 4i\sin^2\frac{\varphi_2}{2}|r|^4\,\dd \beta\wedge \dd \rho$ which
        on $\Omega_\PC$ equals $\frac{i\sin^2\frac{\varphi_2}{2}}{4\sin^4\frac{\beta}{2}}\dd \beta\wedge \dd \rho$.
        The volume form can be written as
        \[\mathrm{Vol}=\frac{4i\sin^2\frac{\varphi_2}{2}|q|^{2(n_2-1)} \det H}{\rho (2i)^{n-1} \mathbf{A}^{n}}
        DY\wedge DZ\wedge \dd \beta\wedge \dd \rho.\]

        The (locally defined) map
        \begin{align*}
            \mathcal{M}_{\varphi, \alpha} & \to \mathcal{M}_{\widehat{\varphi}(\beta), \widehat{\alpha}}
                \times \mathcal{M}_{\widetilde{\varphi}(\beta), \widetilde{\alpha}}\times \CC,
            \\ (y_1,\dots,y_{n_1-1},r,z_1,\dots,z_{n_2-1}) & \mapsto(y_1,\dots,y_{n_1-1}),(z_1,\dots,z_{n_2-1}),r
        \end{align*}
        is locally a homeomorphism for all $\beta_0$ from some open set.


        Note that to find $\iota_v \mathrm{Vol}|_\Omega$ we need only one coordinate of $v$, namely $\dot{\rho}$ where dot stands for differentiation along $v$.
        In affine coordinates $v$ changes the coordinates in the following way.
        \begin{align}
            &\dot{X}_0 =
            \frac{e^{i\varphi_1/2}}{\Im(e^{i\varphi_1/2})}(R-X_0) = (\cot\frac{\varphi_1}{2}+i)(R-X_0),\\
            &\dot{R} =
            \frac{e^{-i\varphi_2/2}}{\Im(e^{i\varphi_2/2})}(X_0-R) = (\cot\frac{\varphi_2}{2}-i)(X_0-R),\\
            &\dot{X}_{n_1} = \dot{X_0}e^{-i(\alpha_1+\dots+\alpha_{n_1})}.
        \end{align}

        Denote $\widehat{\beta} = \alpha_1+\dots+\alpha_{n_1}-\beta$.
        Then, using the relation $r=\frac{1}{1-e^{-i\beta}}$ that holds (locally) on $\Omega$, we obtain
        \begin{align*}
            \dot{r} &= \frac{(\dot{R}-\dot{X}_0)(X_{n_1}-X_0)-(\dot{X}_{n_1} - \dot{X_0})(R-X_0)}{(X_{n_1}-X_0)^2} =\\
                    &= -\left(\cot\frac{\varphi_1}{2}+\cot\frac{\varphi_2}{2}\right)r -
                        \left(e^{-i(\beta+\widehat{\beta})}-1\right)r^2.
        \end{align*}
        By the definition, $\rho = |1 - \frac{1}{r}|^2$, then
        \begin{multline*}
            \dot{\rho} = 2\Re\left(\left(1-\frac{1}{\overline{r}}\right)\frac{\dot{r}}{r^2} \right) =
            2\left(\cot\frac{\varphi_1}{2}+\cot\frac{\varphi_2}{2}\right)
            -2\left( \frac{\cos(\frac{\varphi_1}{2}-\widehat{\beta})}{\sin\frac{\varphi_1}{2}}+
            \frac{\cos(\frac{\varphi_2}{2}-\beta)}{\sin\frac{\varphi_2}{2}} \right)
            =\\=
            -2\frac{\sin\frac{\beta}{2} \sin\frac{\varphi_2-\beta}{2}
                    \sin\frac{\widehat{\beta}}{2} \sin\frac{\varphi_1-\widehat{\beta}}{2}}
                   {\sin\frac{\varphi_1}{2}\sin\frac{\varphi_2}{2}}
            \left(\cot\frac{\beta}{2} +\cot\frac{\varphi_2-\beta}{2}
                  +\cot\frac{\widehat{\beta}}{2} +\cot\frac{\varphi_1-\widehat{\beta}}{2}\right)
        \end{multline*}

        So we have
        \begin{multline*}
            \iota_v \mathrm{Vol}|_\Omega =
            \iota_v \left(\frac{i\sin^2\frac{\varphi_2}{2}|q|^{2(n_2-1)} \det H}
                          {4\sin^4\frac{\beta}{2} (2i)^{n-1} \mathbf{A}^{n}}
            DY\wedge DZ\wedge \dd\beta \wedge \dd\rho\middle)\,\right|_\Omega
            =\\=
            \frac{i\sin^2\frac{\varphi_2}{2}|q|^{2(n_2-1)} \det H\dot{\rho}}
                 {4\sin^4\frac{\beta}{2}(2i)^{n-1} (\mathbf{a}|r|^2)^{n}}
             DY\wedge DZ\wedge \dd \beta
        \end{multline*}

        We can express
        $\mathrm{Vol}_1 = \frac{\det H_1\ DY}{(2i)^{n_1-1}(\mathbf{a}_1|r|^2)^{n_1}}$ and
        $\mathrm{Vol}_2 = \frac{\det H_2\ DZ}{(2i)^{n_2-1}(\mathbf{a}_2|r|^2|q|^{-2)})^{n_2}}$.
        Therefore,
        \[\iota_v \mathrm{Vol}|_\Omega =
        \frac{\det H}{\det H_1\det H_2}
        \frac{\mathbf{a}_1^{n_1}\mathbf{a}_2^{n_2}}{\mathbf{a}^{n}}
        \frac{\sin^2\frac{\varphi_2}{2}}{4\sin^4\frac{\beta}{2}|q|^2}
        \dot{\rho}\ \mathrm{Vol}_1\wedge \mathrm{Vol}_2\wedge \dd\beta.\]

        Using expression for $\det H$ from section~\ref{sec:polygonal_coordinates} we obtain

        \begin{align*}
            \iota_v \mathrm{Vol}|_\Omega =
            \frac{\mathbf{a}_1^{n_1}\mathbf{a}_2^{n_2}}{\mathbf{a}^{n}}
            \frac
            {\sin\frac{\varphi_2}{2}\sin\frac{\varphi_1}{2}}
            {\sin\beta \sin (\varphi_2-\beta) \sin\widehat{\beta} \sin(\varphi_1-\widehat{\beta})}
            \dot{\rho}\ \mathrm{Vol}_1\wedge \mathrm{Vol}_2\wedge \dd\beta
            = \\ =
            \frac{\mathbf{a}_1^{n_1}\mathbf{a}_2^{n_2}}{\mathbf{a}^{n}}
            \left(\cot\frac{\beta}{2} +\cot\frac{\varphi_2-\beta}{2}
                  +\cot\frac{\widehat{\beta}}{2} +\cot\frac{\varphi_1-\widehat{\beta}}{2}\right)
            \mathrm{Vol}_1\wedge \mathrm{Vol}_2\wedge \dd\beta
        \end{align*}

        By the linearity, $\iota_{g(\mathbf{a})v} \mathrm{Vol}|_\Omega$ differs by the factor of $g(a)$, and
        we obtain the required formula.
    \end{proof}

    By the symmetry, the integrand from the vector field $v_2$ will be the same, so Lemma~\ref{integrand} is proved.

    \section{Computation example}\label{sec:examples}

    For $n=1$ the moduli space of 3-punctured sphere is just one point, so the simplest nontrivial case is when $n=2$.
    Consider the most symmetric moduli space $\mathcal{M}_{\pi,\pi, \pi,\pi}$.
    The points of this moduli space correspond to equifacial tetrahedra.
    Let us compute the distribution of $\mathbf{a}$-function.
    By Proposition~\ref{g_integrate} the distribution function $f$ for any smooth function $g$ with compact support
    satisfies
    \[\int_{\mathbb{R}_+} \left(g'(a)-2\frac{g(a)}{a}\right)f(a)ds =
    \int_0 ^{\frac{\pi}{2}} \frac{g(\sin{2\beta})}{\sin{2\beta}}d\beta.\]
    Solving this equation gives
    \[
        f(s) =
        \begin{cases}
            \frac{1-\sqrt{1-s^2}}{s^2}, & \text{for } 0 \leq s \leq 1 \\
            \frac{1}{s^2}, & \text{for } 1 \leq s\text{ .}
        \end{cases}
    \]
    From this answer we can calculate volume of the moduli space:
    $\mathrm{Vol}(\mathcal{M}_{\pi,\pi, \pi,\pi}) = \int_0 ^\infty f(a)da = \frac{\pi}{2}$.


    \begin{figure}
     \centering
     \begin{subfigure}[b]{0.45\textwidth}
         \centering
         \includegraphics[width=\textwidth]{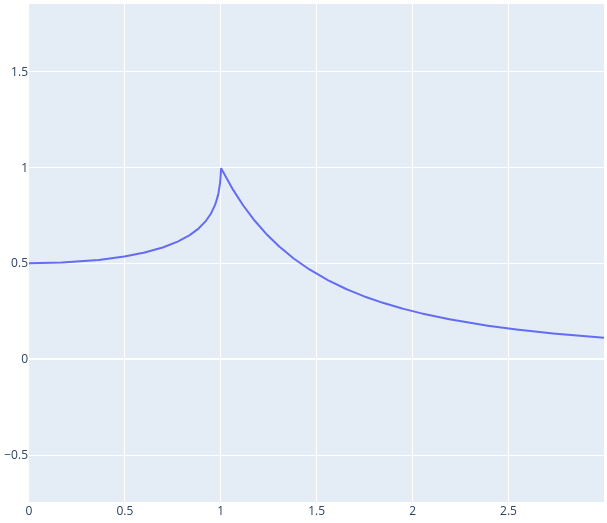}
         \caption{$f(a)$}
         \label{fig:densities_4pi_a}
     \end{subfigure}
     \hfill
     \begin{subfigure}[b]{0.45\textwidth}
         \centering
         \includegraphics[width=\textwidth]{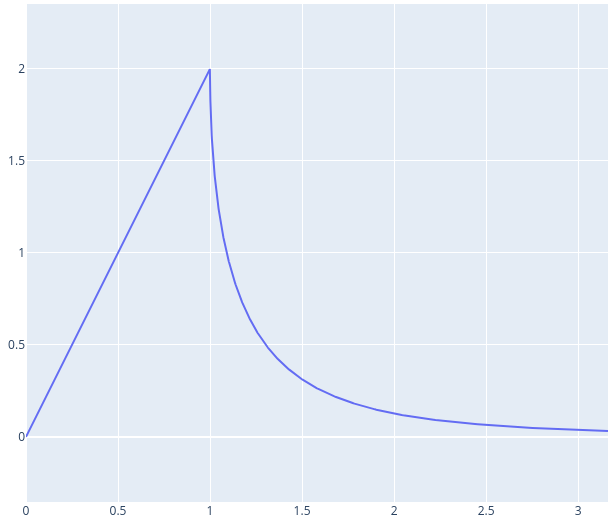}
         \caption{$\rho(l)$}
         \label{fig:densities_4pi_l}
     \end{subfigure}
        \caption{Density of $\bm{a}$- and $\bm{l}$-function for $\varphi=(\pi, \pi), \alpha=(\pi, \pi)$}
        \label{fig:densities_4pi}
    \end{figure}

    We can also transform this into the distribution of the $\bm{l}$-function, the densities of these distributions
    are in Fig.~\ref{fig:densities_4pi}.
    From numeric computation, the mean and the median values of the length for equifacial tetrahedra of unit area
    are approximately $1.09$ and $0.89$.
    For comparison, the value of $\bm{l}$-function for the regular tetrahedron is $1/\sqrt[4]{3}\approx 0.76$
    and for the double-sided surface of a square (which also belongs to this moduli space) is $1/\sqrt{2}\approx0.71$
    if marked singular points form a side of the square and $1$ if the marked points are opposite.

    \begin{figure}
     \centering
     \begin{subfigure}[b]{0.45\textwidth}
         \centering
         \includegraphics[width=\textwidth]{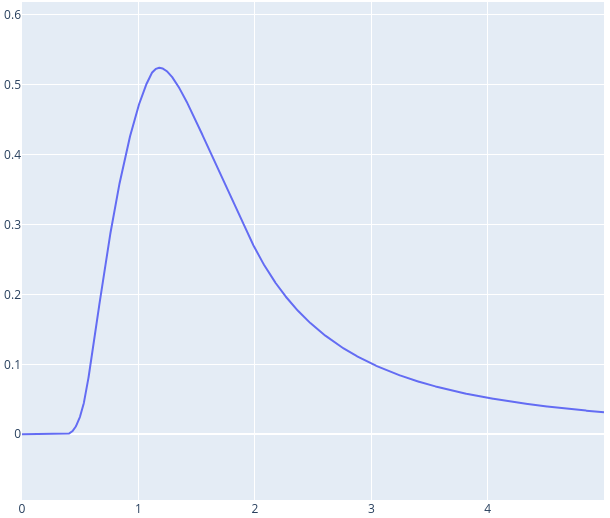}
         \caption{$f(a)$}
         \label{fig:densities_54pi5_a}
     \end{subfigure}
     \hfill
     \begin{subfigure}[b]{0.45\textwidth}
         \centering
         \includegraphics[width=\textwidth]{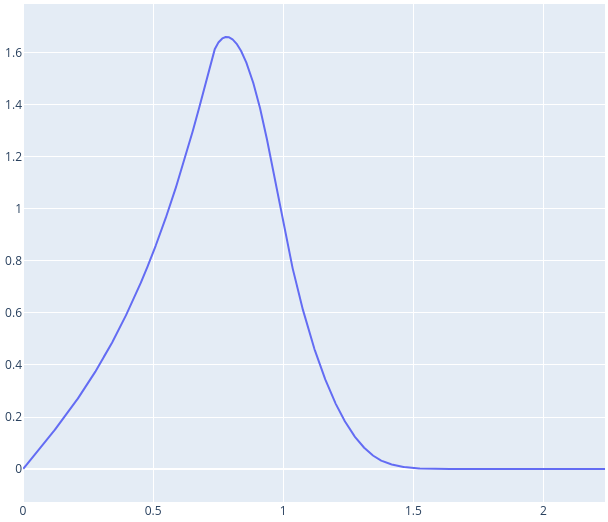}
         \caption{$\rho(l)$}
         \label{fig:densities_54pi5_l}
     \end{subfigure}
        \caption{Density of $\bm{a}$- and $\bm{l}$-function for $\varphi = (6\pi/5, 6\pi/5), \alpha = (4\pi/5, 4\pi/5, 4\pi/5)$}
        \label{fig:densities_54pi5}
    \end{figure}

    We also give the results of numerical computations for distributions of area and length functions
    in case $\varphi = (6\pi/5, 6\pi/5), \alpha = (4\pi/5, 4\pi/5, 4\pi/5)$, i.e.~five cone singularities
    of equal total angle.
    The plots are in Figure~\ref{fig:densities_54pi5}, the mean and median of $\bm{l}$ are approximately
    $0.71$ and $0.76$.
    For comparison, for double-sided regular pentagon the $\bm{l}$-function is approximately $0.54$ if
    the marked points are chosen to form a side of the pentagon, and approximately $0.87$ if
    the marked points form a diagonal.
    For the square pyramid with required angle defects the $\bm{l}$-function is approximately $0.45, 0.64, 0.70$
    for different choices of marked points.

    \bibliographystyle{plain}
    \bibliography{refs}

\end{document}